\documentclass{amsart}
\author{Julien Melleray}
\address{Universit\'e Claude Bernard -- Lyon 1 \\
  Institut Camille Jordan, CNRS UMR 5208 \\
  43 boulevard du 11 novembre 1918 \\
  69622 Villeurbanne Cedex \\
  France}
\usepackage{amssymb}
\usepackage[initials,shortalphabetic]{amsrefs}
\usepackage[all]{xy}
\usepackage{mathpazo}
\usepackage{hyperref}
\usepackage{my-macros}
\numberwithin{equation}{section}

\title{Dynamical simplices and Borel complexity of orbit equivalence}

\begin{document}

\begin{abstract}
We prove that any divisible dynamical simplex is the set of invariant measures of some Toeplitz subshift. We apply our construction to prove that orbit equivalence of Toeplitz subshifts is Borel bireducible to the universal equivalence relation induced by a Borel action of $S_\infty$.

\end{abstract}
\maketitle

\section{Introduction}

This paper is a continuation of earlier work of the author (in part joint with T. Ibarluc\'ia) concerning sets of invariant probability measures of minimal homeomorphisms of the Cantor space, and an application of this work to the computation of the complexity, in the sense of \emph{Borel reducibility} theory, of the relation of orbit equivalence of minimal homeomorphisms (more precisely, of \emph{Toeplitz subshifts}).

Recall that a homeomorphism of a Cantor space is \emph{minimal} if all of its orbits are dense. Minimal homeomorphisms are a classical object of study in topological dynamics; let us briefly discuss an important example. Given a finite alphabet $A$, the shift map $S \colon A^\Z \to A^\Z$ is defined by setting $S(x)(n)=x(n+1)$. Of course $S$ is not minimal, but there exist minimal (in the sense of inclusion) closed $S$-invariant subsets $F$ of $A^\Z$; whenever such an $F$ is infinite it must be homeomorphic to the Cantor space, and the restriction of $S$ to $F$ is called a \emph{minimal subshift}. An even more specific example is provided by \emph{Toeplitz subshifts}, which are briefly discussed below (Section \ref{s:notations}).

Any homeomorphism $g$ of a Cantor space $X$ induces an equivalence relation $R_g$, whose classes are the $g$-orbits. Homeomorphisms with very different dynamical properties may induce the same equivalence relation, and one is led to the following notion: two homeomorphisms $g,h$ of $X$ are \emph{orbit equivalent} if there exists a homeomorphism $f$ of $X$ such that 
$$\forall x,x'\in X \quad \left(x R_g x'\right) \leftrightarrow (f(x)R_h f(x'))\ . $$
Any homeomorphism $g$ of a Cantor space must have a nonempty set of invariant Borel probability measures, which we denote by $K_g$; furthermore, a map witnessing that $g,h$ are orbit equivalent must push forward $K_g$ onto $K_h$. The fact that the converse holds for minimal homeomorphisms is much more surprising, and is the content of a celebrated theorem of Giordano, Putnam and Skau.

\begin{theorem*}[Giordano--Putnam--Skau \cite{Giordano1995}]
Let $g,h$ be two minimal homeomorphisms of a Cantor space $X$, and denote by $K_g,K_h$ their sets of invariant Borel probability measures. Then the following conditions are equivalent.
\begin{itemize}
\item The homeomorphisms $g,h$ are orbit equivalent.
\item There exists a homeomorphism $f$ of $X$ such that $f_*K_g=K_h$.
\end{itemize}
\end{theorem*}

This naturally led to an investigation of sets of invariant measures of minimal homeomorphisms, which was bolstered by the following result, itself an immediate consequence of a theorem of Glasner--Weiss.

\begin{theorem*}[see \cite{Glasner1995a}*{Lemma 2.5}]
Let $g$ be a minimal homeomorphism of a Cantor space $X$, and denote again by $K_g$ the set of $g$-invariant Borel probability measures on $X$. Assume that $A,B$ are clopen subsets of $X$ such that $\mu(A)< \mu(B)$ for any $\mu \in K$. Then there exists a clopen subset $C$ of $X$ such that $C$ is contained in $B$, and $\mu(C)=\mu(A)$ for all $\mu \in K$.
\end{theorem*}

Now we turn the problem on its head: we start from a set $K$ of probability measures on a Cantor space $X$, and ask under which conditions there exists a minimal homeomorphism $g$ of $X$ such that $K=K_g$. The above theorem of Glasner--Weiss imposes a strong, nontrivial necessary condition even in the case when $K$ is a singleton; it turns out that this condition, along with some obvious other necessary conditions, is necessary, a fact that was established by the author \cite{Melleray2018}, following joint work with T. Ibarluc\'ia in which the approach was laid out and a slightly weaker result was obtained.

\begin{theorem*}[\cite{Ibarlucia2016}*{Theorem~1.1} and \cite{Melleray2018}*{Theorem~2.1}]
Let $X$ be a Cantor space, and $K$ be a subset of the (compact) space of probability measures on $X$. Then there exists a minimal homeomorphism $g$ of $X$ such that $K=K_g$ if, and only if, the following conditions are satisfied:
\begin{enumerate}
\item $K$ is compact and convex.
\item Every element of $K$ is nonatomic and has full support.
\item $K$ satisfies the \emph{Glasner--Weiss condition}: for every clopen $A,B$ in $X$, if $\mu(A)< \mu(B)$ for every $\mu \in K$ then there exists 
a clopen $C \subset B$ such that $\mu(C)=\mu(A)$ for every $\mu \in K$.
\end{enumerate}
\end{theorem*}

In addition to the aforementioned papers, the above result was also preceded by work of Akin \cite{Akin2005}, who established the same theorem when $K$ is a singleton; and Dahl \cite{Dahl2008}, who extended this to a more general setting which applies in particular to all compact, convex sets of probability measures with finitely many extreme points. Following Dahl, we say that $K$ is a \emph{dynamical simplex} if it satisfies the conditions of the previous theorem.

The author then noticed a connection of this result to \emph{Fra\"iss\'e theory} (a quick discussion, and some references, are given in the last section of this paper, as well as in \cite{Melleray2018}) and used this to produce a new, rather elementary, proof of the following theorem, which is a particular case of a theorem of Downarowicz \cite{Downarowicz1991}.

\begin{theorem*}
Let $K$ be a nonempty metrizable Choquet simplex. There exists a minimal homeomorphism $g$ of a Cantor space $X$ such that $K$ is affinely homeomorphic to $K_g$.
\end{theorem*}

Downarowicz proved a more precise result, namely he showed that $g$ above can be taken to be a Toeplitz subshift of $\{0,1\}^\Z$. In order to obtain this result via the approach of \cite{Melleray2018}, one needs to understand which dynamical simplices are sets of invariant measures for (homeomorphisms conjugate to) Toeplitz subshifts.We establish here a simple sufficient condition; say that a dynamical simplex $K$ on a Cantor space $X$ is \emph{divisible} if for any clopen $A$ of $K$ and any integer $n$ there exists a clopen $B \subset A$ such that $\mu(B)=\frac{\mu(A)}{n}$ for all $\mu \in K$.

\begin{theorem*}
Let $X$ be a Cantor space, and $K$ a divisible dynamical simplex on $X$. Then there exists a homeomorphism $g$ of $X$ which is isomorphic to a Toeplitz subshift on $\{0,1\}^\Z$ and is such that $K=K_g$.
\end{theorem*}

This theorem, along with the arguments of \cite{Melleray2018}, enables one to realize any nonempty metrisable Choquet simplex as the set of invariant Borel probability measures of a Toeplitz subshift. An important fact for our purposes is that this construction is rather flexible, and this flexibility enables us to compute the \emph{complexity}, in the sense of \emph{Borel reducibility theory}, of the relation of orbit equivalence of Toeplitz subshifts. We say a few words about this theory before stating our result.

Let $R,S$ be two equivalence relations on standard Borel spaces $X,Y$; one says that $R$ \emph{Borel reduces} to $S$ if there exists a Borel map $f \colon X \to Y$ such that
$$\forall x,x' \in X \quad \left(x R x'\right) \Leftrightarrow \left(f(x) S f(x')\right)\ . $$
The idea is that $f$ realizes a ``computable'' embedding of $X/R$ into $Y/S$; said differently, $f$ reduces, in a computable manner, the problem of understanding when two points are $R$-equivalent (the \emph{classification problem} associated to $R$) to the problem of understanding when two points are $S$-equivalent. Given two relations $R,S$ as above, one says that they are \emph{Borel bireducible} if each Borel reduces to the other; intuitively, one then considers the complexities of the corresponding classification problems to be the same.

The theory of Borel reducibility was introduced by Friedman and Stanley in \cite{Friedman1989}, and is by now a rich and well-developed area (see for instance \cite{Hjorth2000}, \cite{Gao2009a} and references therein). An important point is that for any Polish group $G$, there exists an equivalence relation which arises from a Borel $G$-action on a standard probability space and is such that any other such equivalence relation Borel reduces to it. Clearly this relation is unique up to Borel bireducibility, and we commit an abuse of notation by calling it \emph{the} universal relation arising from a Borel action of $G$, and denote it by $E_G$. If $H$ is a closed subgroup of $G$, then $E_H$ Borel reduces to $E_G$.

A particularly important, and ubiquitous, Polish group is the group $S_\infty$ of all permutations of the integers; many natural equivalence relations happen to be Borel bireducible to $E_{S_\infty}$ (for instance, isomorphism of countable graphs, countable linear orderings, countable fields may all be seen as sitting at this particular complexity level). The specific equivalence relation that will play a role in our construction is the relation of homeomorphism between closed subsets of the Cantor space; its complexity was computed by Camerlo and Gao.

\begin{theorem*}[\cite{Camerlo2001}*{Theorem~3}]
The relation of homeomorphism between closed subsets of the Cantor space is (Borel bireducible to) the universal equivalence relation induced by a Borel action of $S_\infty$.
\end{theorem*}

It is natural to ask what exactly is the complexity of the relation of orbit equivalence of minimal homeomorphisms of the Cantor space. The theorem of Giordano--Putnam--Skau mentioned above essentially asserts that this relation is (Borel) reducible to the relation of \emph{isomorphism} of dynamical simplices, which is defined as one would expect: two dynamical simplices $K,L$ on the Cantor space $X$ are isomorphic if there exists a homeomorphism $g$ of $X$ such that $g_*K=L$. This relation is induced by a continuous action of the homeomorphism group $\Homeo(X)$, and this group is isomorphic to a closed subgroup of $S_\infty$ (think of a homeomorphism as acting by permutation on the countable set made up of all clopen subsets of $X$). Thus isomorphism of dynamical simplices sits below $E_{S_\infty}$ in terms of Borel complexity. 

Conversely, it follows from the construction of \cite{Ibarlucia2016} (improved here in order to obtain a Toeplitz subshift) and another application of the Giordano--Putnam--Skau theorem, that the relation of isomorphism of divisible dynamical simplices Borel reduces to the relation of orbit equivalence of Toeplitz subshifts. Using this approach, we establish the following result.

\begin{theorem*}
The relation of homeomorphism between closed subsets of the Cantor space Borel reduces to the relation of isomorphism of divisible dynamical simplices. Consequently, the relation of orbit equivalence of Toeplitz subshifts is (Borel bireducible to) the universal equivalence relation induced by a Borel action of $S_\infty$. 
\end{theorem*}

The paper is organized as follows. We first give more background on some of the facts mentioned in the introduction; then we indicate how one can modify the construction of \cite{Ibarlucia2016} in order to prove that a divisible dynamical simplex is the set of invariant measures of some $\{0,1\}$-Toeplitz subshift. Once that task is complete we do some bookkeeping, checking that various sets and maps are Borel, and proving that isomorphism of dynamical simplices and orbit equivalence of minimal homeomorphisms are Borel bireducible equivalence relations. Finally, we discuss the construction of \cite{Melleray2018} and explain how it can be used to produce a Borel reduction from the relation of homeomorphism of closed subsets of the Cantor space to the relation of orbit equivalence of Toeplitz subshifts.

{\bf Acknowledgements.} Work on this paper was initiated during a BIRS-CMO workshop in Oaxaca during the summer of 2017; revived while visiting IMPAN in Warsaw in Spring 2018; and a final technical issue was fixed while attending a conference at the Kurt G\"odel research center in Vienna at the end of the summer of 2018. I am grateful to the staff of all these places for their kind hospitality and the excellent working conditions they provided. Thanks are also due to Maciej Malicki and Andrew Zucker for useful conversations. Work of the author was partially supported by ANR projects GAMME (ANR-14-CE25-0004) and AGRUME (ANR-17-CE40-0026).

\section{Background}

\subsection{Notations and first definitions}\label{s:notations}
Given a Cantor space $X$, $\Clopen(X)$ denotes the countable Boolean algebra made up of all clopen subsets of $X$.  By Stone duality, a homeomorphism of $X$ corresponds uniquely to an automorphism of $\Clopen(X)$, and conversely any such automorphism is associated to a unique homeomorphism of $X$.

The space of Borel probability measures on $X$, which we denote by $\Prob(X)$, has a natural compact topology induced by the maps $\mu \mapsto \mu(U)$, where $U$ ranges over all clopen subsets of $X$; it is worth keeping in mind that a Borel probability measure on $X$ is uniquely determined by its values on clopen sets, and that any finitely additive probability measure on $\Clopen(X)$ extends to an element of $\Prob(X)$.

The group $\Homeo(X)$ of all homeomorphisms of $X$ is a Polish group, and for our purposes its topology is best understood by thinking of homeomorphisms as automorphisms of $\Clopen(X)$; a basis of neighborhoods of identity is made up of sets of the form
$\{g \in \Homeo(X) \colon \forall i \in \{0,\ldots,n\} \ g(U_i)=U_i\} $
where $(U_0,\ldots,U_n)$ ranges over all finite families of clopen subsets of $X$. 

We recall that a homeomorphism of $X$ is \emph{minimal} if each of its orbits is dense.

Let $A$ be a finite alphabet, and $X_A=A^\Z$. Then one may consider the shift map $S \colon X_A \to X_A$ defined by $S(x)(n)= x(n+1)$ for all $x \in X_A$ and all $n \in \Z$. This is a homeomorphism, which is clearly not minimal; a \emph{minimal subshift} is an infinite, closed, $S$-invariant subset of $X_A$ on which $S$ acts minimally. Those always exist, and of particular relevance to us are the \emph{Toeplitz subshifts}.

\begin{defn}
Let $A$ be a finite alphabet. A sequence $x \in A^\Z$ is \emph{quasiperiodic} if for all $m \in \Z$ there exists $p \ne 0$ such that $x(m)=x(m+np)$ for all $n \in \Z$ (equivalently, every finite subword of $x$ occurs periodically in $x$). The sequence is \emph{Toeplitz} if it is quasiperiodic but not periodic.

If $x$ is Toeplitz, the associated \emph{Toeplitz subshift} is the closure of the orbit of $x$ in $X_A$ under the shift action. 
\end{defn}

Note for future use that $x$ is quasiperiodic if and only if for any neighborhood $U$ of $x$ there exists $p>0$ such that $S^{np}(x) \in U$ for all $n \in \Z$ (this is the topological reformulation of the fact that each subword occurs periodically; of course the period depends on the subword).
It is not hard to check that any Toeplitz subshift is minimal, though one has to pay attention to the fact that not every element of a Toeplitz subshift is a Toeplitz sequence. For more information on Toeplitz subshifts we refer to \cite{Williams1984}.

\subsection{Kakutani--Rokhlin partitions}
Our approach to minimal homeomorphisms is via \emph{Kakutani--Rokhlin partitions}, and we review the basics now. In this subsection, we fix a minimal homeomorphism $\varphi$ and a Cantor space $X$. First, notice that for any nonempty open set $O$ one must have 
$X=\bigcup_{n \in \Z}  \varphi^n O$ by minimality, so by compactness there exists $N$ such that 
$$X=\bigcup_{n=-N}^{-1} \varphi^n O \ .$$
Thus one may define the \emph{first return map} associated to $O$: for any $x \in O$, let $n(x)=\min\{n \ge 1 \colon \varphi^n(x) \in O\}$ and set $\varphi_O(x)=\varphi^{n(x)}(x)$. When $O$ is assumed to be clopen, the map $n(x)$ is continuous, and $\varphi_O$ is easily checked to be a homeomorphism of $O$. Let $\{n(x) \colon x \in O\}$ be enumerated as $n_1,\ldots,n_k$, and for every $i$ set $O_i=\{x \in O \colon n(x)=n_i\}$. Then for all $i$ and all $j < n_i$ define $O_{i,j}= \varphi^j(i)$. Then the following conditions are satisfied:
\begin{enumerate}
\item The family $\mcO=(O_{i,j})$ forms a clopen partition of $X$.
\item For every $i$ and every $j<n_i$ one has $\varphi(O_{i,j})=O_{i,j+1}$. 
\end{enumerate}
These conditions form the definition of a Kakutani--Rokhlin partition associated to $\varphi$; the \emph{base} of the partition is the set $\bigcup_{i} O_{i,0}$, while its \emph{top} is $\bigcup_{i} O_{i,n_i}$. by a \emph{column} of a Kakutani--Rokhlin partition $\mcO$, we mean a set of the form $\{O_{i,j} \colon 0 \le j \le n_i\}$. We say that $n_i+1$ is the \emph{height} of that column. 

\begin{defn}
Let $\mcO$ and $\mcU$ be two Kakutani--Rokhlin partitions. We say that $\mcU$ \emph{refines} $\mcO$ if the base of $\mcU$ is contained in the base of $\mcO$, and every element of $\mcO$ is a union of elements of $\mcU$.
\end{defn}

Whenever $\mcU$ refines $\mcO$, $\mcU$ has been obtained from $\mcO$ by \emph{cutting and stacking}; that is, the base of $\mcU$ is endowed with a partition finer than that induced from $\mcO$, and the columns of $\mcU$ are obtained by stacking small slices of the columns of $\mcO$ on top of each other. Indeed, any element of $\mcU$ is contained in a unique element $O(U)$ of $\mcO$ and, if $O(U)$ does not belong to the top, it is mapped by $\varphi$ one level up in $\mcO$, so that $\varphi(U)$ is contained in $\varphi(O(U))$. However, when $O(U)$ belongs to the top of $\mcO$, there is no way to guess from $\mcO$ where $O(U)$ will be mapped to.

A particularly useful fact is the possibility of cutting along the columns: let $\mcO$ be a Kakutani--Rokhlin partition, and $A$ a clopen subset of $X$. For any $i$ and $j \le n_i$, let $U_{i,k}$ enumerate the atoms of the partition generated by the clopen sets $\varphi^{-j}(A \cap U_{i,j})$ and $\varphi^{-j}(U_{i,j} \setminus A)$. Then one can form a new Kakutani--Rokhlin partition, with the same base, with columns enumerated by all $U_{i,k}$ and levels $U_{i,k,j}=\varphi^j(U_{i,k})$; each column of $\mcO$ has been cut vertically to produce new, thinner columns with the same height as the original one. After this operation has been performed, the new Kakutani--Rokhlin partition $\mcU$ is \emph{compatible} with $A$, that is, $A$ is a union of elements of $\mcU$.

Given a Kakutani--Rokhlin partition $\mcO$, let $\mcO_{base}$ denote the subalgebra of $\Clopen(X)$ with atoms $\{O_{i,j} \colon 0 <j \le n_i\} \cup \{\text{base}(\mcO)\}$, and $\mcO_{top}$ the subalgebra with atoms $\{O_{i,j} \colon 0 \le j < n_i\} \cup \{\text{top}(\mcO)\}$. Then $\varphi$ induces a partial automorphism $\varphi_\mcO \colon \mcO_{top} \to \mcO_{base}$ of the Boolean algebra $\Clopen(X)$, such that $\varphi(O_{i,j})=O_{i,j+1}$ whenever $j< n_i$, and $\varphi(\text{top}(\mcO))=\text{base}(\mcO)$. In turn, this partial automorphism defines an open neighborhood $U_\mcO$ of $\varphi$ in $\Homeo(X)$, which consists of all homeomorphisms of $X$ which extend $\varphi_\mcO$. 

Now, fix $x \in X$, and consider a sequence of Kakutani--Rokhlin partitions $(\mcO_n)$ such that 
\begin{enumerate}
\item $\mcO_{n+1}$ refines $\mcO_n$ for all $n$.
\item $(\mcO_n)$ generates $\Clopen(X)$ (i.e. any clopen set is a union of elements of some $\mcO_n$).
\item The intersection of the bases of $\mcO_n$ is equal to $\{x\}$ (and then the tops must intersect to $\varphi^{-1}(x)$).
\end{enumerate}
The fact that these conditions can be satisfied is a simple consequence of the countability of $\Clopen(X)$ and the possibility of cutting columns as explained above. Under these conditions, the intersection of the open neighborhoods $U_{\mcO_n}$ is equal to $\{\varphi\}$; that is, knowing the behavior of $\varphi$ on a sufficiently rich sequence of Kakutani--Rokhlin partitions is sufficient to reconstruct $\varphi$ entirely (essentially, these partitions encode a neighborhood basis of $\varphi$).

\subsection{Invariant measures of minimal homeomorphisms and a theorem of Glasner--Weiss}

Given a minimal homeomorphism $\varphi$ of the Cantor space $X$, we denote $K_\varphi=\{ \mu \in \Prob(X) \colon \varphi_* \mu = \mu\}$, and 
$G_\varphi= \{g \in \Homeo(X) \colon \forall \mu \in K_\varphi \ g_* \mu= \mu\}$. By definition $\varphi$ belongs to $G_\varphi$, and $G_\varphi$ is a subgroup of $\Homeo(X)$; as it turns out, it follows from a result of Glasner--Weiss that $G_\varphi$ is the closure of the \emph{full group} of $\varphi$, see \cite{Glasner1995a}, \cite{Ibarlucia2016}. We state two consequences of that result that will be useful to us.

\begin{theorem}[Glasner--Weiss \cite{Glasner1995a}*{Lemma~2.5 and Proposition~2.6}]\label{t:glasnerweiss}
Let $\varphi$ be a minimal homeomorphism of the Cantor space $X$, and $A,B$ two clopen subsets of $X$. 
\begin{enumerate}
\item Assume that $\mu(A)< \mu(B)$ for all $\mu \in K_\varphi$. Then there exists $g \in G_\varphi$ such that $g(A) \subset B$.
\item Assume that $\mu(A)=\mu(B)$ for all $\mu \in K_\varphi$. Then there exists $g \in G_\varphi$ such that $g(A)=B$.
\end{enumerate}
\end{theorem}

\section{Divisible dynamical simplices and Toeplitz subshifts}

\begin{defn}
Let $X$ be a Cantor space, and $K$ a subset of $\Prob(X)$. We say that $K$ is a \emph{dynamical simplex} if $K$ satisfies the following conditions:
\begin{enumerate}
\item $K$ is nonempty, compact and convex.
\item All elements of $K$ are atomless and have full support.
\item For any clopen $U,V \in K$ such that $\mu(U)< \mu(V)$ for all $\mu \in K$, there exists a clopen $W \subset V$ such that $\mu(U)=\mu(W)$ for all $\mu \in K$ (We then say that $K$ has the \emph{Glasner--Weiss property}).
\end{enumerate}
\end{defn}

We say that a subset $K$ of $\Prob(X)$ is \emph{divisible} if it is true that, for any clopen $U$ and any $n \in \N^*$, there exists a clopen $V \subseteq U$ such that $\mu(V)=\frac{1}{n} \mu(U)$ for all $\mu \in K$. Note that not all dynamical simplices are divisible (though they all satisfy a weaker, closely related condition, see \cite{Melleray2018}*{Theorem~2.1}). 

It was proved in \cite{Melleray2018} (following \cite{Ibarlucia2016}, itself continuing and extending earlier work of Akin \cite{Akin2005} and Dahl \cite{Dahl2008}) that $K$ is a dynamical simplex if and only if there exists a minimal homeomorphism $\varphi$ such that $K$ is equal to the set of all $\varphi$-invariant probability measures; this explains the terminology ``dynamical simplex'', which was introduced by Dahl. Our aim in this section is to prove the following variant of that result.

\begin{theorem}\label{t:Toeplitz}
Assume that $K$ is a \emph{divisible} dynamical simplex. Then there exists a $\{0,1\}$ Toeplitz subshift $\varphi$ such that 
$$K=\{\mu \in \Prob(X) \colon \varphi_* \mu=\mu\}\ . $$
\end{theorem}

Since the dynamical simplices built in \cite{Melleray2018} are divisible, this result, combined with \cite{Melleray2018}, recovers Downarowicz's result that any nonempty metrizable Choquet simplex is affinely homeomorphic to the space of invariant measures of a $\{0,1\}$ Toeplitz subshift.

In order to prove Theorem \ref{t:Toeplitz}, we need to refine the argument of \cite{Ibarlucia2016}, using the fact that $K$ is divisible to ensure some additional conditions. We first recall some notions and notations.

\begin{defn}
Let $K$ be a dynamical simplex.  Given $U,V \in \Clopen(X)$, we denote $U\sim_K V$ when it is true that
$$\forall \mu \in K \quad \mu(U)=\mu(V) \ .$$
In that case, we often write that $U$ and $V$ have equal measures.
\end{defn}

Denote by $G_K$ the group $\{g \in \Homeo(X) \colon \forall \mu \in K \ g_*\mu=\mu\}$. Assuming that $K$ is a dynamical simplex, it follows from Theorem \ref{t:glasnerweiss} that $U \sim_K V$ holds if and only if there exists $g \in G_K$ such that $g(U)=V$.

\begin{defn}[\cite{Ibarlucia2016}]
Let $K$ be a dynamical simplex, and $U \in \Clopen(X)$. A \emph{KR-partition} (associated to $K$) of $U$ is a partition of $U$ in clopen subsets $U_{i,j}$, where $i \in \{1,\ldots,n\}$ and $j \in \{0,\ldots,n_{i-1}\}$ such that 
$$\forall i \ \forall j,k \in \{0,\ldots,n_i\} \quad U_{i,j} \sim_K U_{i,k}\ . $$
The union of all $U_{i,j}$ for a fixed $i$ is called a \emph{column} of this partition, and $n_i+1$ is the \emph{height} of that column.
\end{defn}

This definition is modelled on that of a Kakutani--Rokhlin partition, except that there is no named homeomorphism mapping one level of a column to the next. There is an obvious notion of refinement for KR-partitions, introduced in \cite{Ibarlucia2016}, and we again use the terminology of \emph{cutting} and \emph{stacking}. If $\mcT$ refines $\mcS$, the columns of $\mcT$ consist of small slices of columns of $\mcS$ stacked onto each other; we use the terminology \emph{copies} to describre these slices. For instance, given any column $C$ of $\mcS$, we will often mention the copies of $C$ contained in a given column of $\mcT$.

To each KR-partition a natural partial automorphism is associated, which maps each $U_{i,j}$ to $U_{i,j+1}$ for $j < n_i$, and maps $\bigcup_i U_{i,j_i}$ (the \emph{top} of the partition) to $\bigcup U_{i,0}$ (the \emph{base} of the partition). 
Associated to any KR-partition, there is an open subset in $\Homeo(X)$, made up of all homeomorphisms extending the partial automorphism associated to the KR-partition at hand. A sequence of partitions refining each other thus induces a nested sequence of open subsets which, under appropriate conditions, intersect in a singleton consisting of a minimal homeomorphism. This is how the construction of \cite{Ibarlucia2016} proceeds; we use the same basic idea here, but need to ensure some additional conditions to obtain a Toeplitz subshift in the end. This additional work is based on a lemma which we discuss now; most of the work is done to ensure that there exists a finite (indeed, with $2$ elements) clopen generating partition  
for the homeomorphism obtained at the end of the construction.

\subsection{A refinement lemma}
In this subsection, we fix a divisible dynamical simplex $K$ and a $KR$-partition $\mcT$; we denote the columns of $\mcT$ by $C_1,\ldots,C_n$ and assume that $n \ge 2$. 

\begin{defn}
Let $\mcS$ be a KR-partition refining $\mcT$. Given a column $C$ of $\mcS$, and $i \in \{1,\ldots,n\}$, denote by $k_i$ the number of copies of $C_i$ contained in $C$; the \emph{repartition} of $C$ is the vector $(\frac{k_1}{\text{height}(C)},\ldots,\frac{k_n}{\text{height}(C)})$

Two columns of $\mcS$ are said to be $\mcT$-\emph{twins} if they have the same repartition.
\end{defn}

\begin{lemma}\label{l:refinement}
Assume that $\mcS$ is a KR-partition refining $\mcT$, and that at least two columns of $\mcS$ are not $\mcT$-twins. Then one may further refine $\mcS$ to a KR-partition $\mcS'$ such that:
\begin{enumerate}
\item Each column of $\mcS'$ contains at least one copy of each column of $\mcT$.
\item No two columns of $\mcS'$ are $\mcT$-twins. 
\item All the columns of $\mcS'$ have the same height.
\end{enumerate}
\end{lemma}

\begin{proof}
We begin by proving the first part of the assertion. Assume that $D$ is a column of $\mcS$ which does not contain a copy of $C_1$ (say). By cutting the base of $D$ into $N$ pieces of equal measures, then cutting vertically to form $N$ smaller columns and stacking those on top of each other, we may assume that the top of $D$ is small enough that one can map it (via a homeomorphsism preserving all measures in $K$) into one column of $\mcS$ which contains a copy of $C_1$, obtaining a refinement of $\mcS$ with the same number of columns as $\mcS$ and one less column not containing a copy of $C_1$. By choosing $N$ large enough, one can also make the repartition of the new column of $\mcS$ arbitrarily close to (but necessarily different from) the repartition of the original column, ensuring that at least two columns of this new KR-partition are not twins. Repeating this operation as necessary, we find a refinement of $\mcS$ satisfying the first item above. This will also hold true of any partition which refines it, so to simplify notation we may as well assume that $\mcS$ already satisfies that condition.

To ensure that the second condition holds, let $m(\mcS)$ be the total number of columns of $\mcS$ which have a $\mcT$-twin. If $m(\mcS)=0$ we have nothing to do; otherwise, reasoning inductively, it is enough to prove that there is $\mcS'$ refining $\mcS$ and such that $m(\mcS')< m(\mcS)$.
So, assume that $m(\mcS) \ge 2$; let $v_1,\ldots,v_p$ enumerate the repartition vectors associated to the elements of $\mcS$ (i.e.~ $v_i \ne v_j$ if $i \ne j$), and assume that there are two columns with repartition $v_1$. Pick such a column $C$, and choose also a column $D$ which is not a twin of $C$, say with repartition $v_2$. We apply the same trick as before: cut the base of $C$ into a large number $N$ of small pieces with the same measures, cut vertically to obtain $N$ smaller copies of $C$, and stack those on top of each other (note that this does not affect the repartition of these columns). We thus reduce to the case where the measure of the base of $C$ is strictly less than the measure of the base of $D$. Then consider the new $KR$-partition obtained by stacking a copy of $D$ on top of $C$, and leaving all other columns unchanged; by choosing $N$ very large, one can make the repartition of the new column arbitrarily close to that of $C$, thus different from $v_2,\ldots,v_p$; since $C$ and $D$ are not $\mcT$-twins, this repartition is also different from $v_1$, and we are done.

Once we have found $\mcS'$ satisfying the first two conditions above, let $D_1,\ldots,D_p$ denote the columns of $\mcS'$, with heights $h_1,\ldots,h_p$. Let $q$ be a common multiple of $h_1,\ldots,h_p$, and $q_i= \frac{q}{h_i}$. Using the fact that $K$ is divisible, one may cut the base of each $D_i$ into $q_i$ pieces of equal measures, and then stack those thinner columns on top of each other. One then obtains a new partition, with $p$ columns having the same repartitions as the columns of $\mcS'$ (in particular, no two columns are $\mcT$-twins), and such that the height of each column is equal to $q$.
\end{proof}

\subsection{Proof of Theorem \ref{t:Toeplitz}}
We fix a compatible metric on the Cantor space $X$. Assume that $K$ is a divisible dynamical simplex. Our construction is based on two propositions, which are simple variants of results from \cite{Ibarlucia2016}; hoping to shorten the exposition a bit, we use these results as blackboxes here. Their proofs are very similar in spirit to what we are doing here - cutting and stacking as needed in order to produce KR-partitions with good properties. 

\begin{lemma}\label{prop:towerdiameter}
Fix a KR-partition $\mcT$, with at least two columns, and $\varepsilon >0$. there exists a KR-partition $\mcS$ such that
\begin{enumerate}
\item $\mcS$ refines $\mcT$.
\item The base and top of $\mcS$ both have diameter less than $\varepsilon$.
\item All columns of $\mcS$ have the same height, each of them contains at least one copy of every column of $\mcT$, and no two of them are $\mcT$-twins.
\end{enumerate}
\end{lemma}

\begin{proof}
Denote the columns of $\mcT$ by $C_1,\ldots,C_n$. By cutting $C_1$ if necessary, we may assume that both its top and its base have diameter less than $\varepsilon$. Next, cut $C_1$ into two nonempty columns $C_0'$ and $C_1'$, and let $Y$ be the union of $C_1',C_2,\ldots,C_n$ and $\mcT'$ the KR-partition of $Y$ with columns $C_1',\ldots,C_n$. Using the same idea as before, we refine $\mcT'$ into a KR-partition $\mcS_Y$ of $Y$ such that no column of $\mcS_Y$ consists entirely of copies of $C_1'$. 
Then, applying the argument of \cite{Ibarlucia2016}*{Proposition 3.4} to $\mcS_Y$, we further refine it to a KR-partition $\mcS_Y'$ whose base and top are contained in the base and top of $C_1'$. By adjoining $C_0'$ to $\mcS_Y'$, we obtain a KR-partition of $X$ refining $\mcT$, whose base and top have diameter less than $\varepsilon$ and such that at least two of its columns are not $\mcT$-twins (by construction $C_0'$ has no $\mcT$-twin in $\mcS_Y'$). We then conclude by applying Lemma \ref{l:refinement}.
\end{proof}

We recall that a KR-partition $\mcT$ is \emph{compatible} with a clopen set $U$ if $U$ belongs to the Boolean algebra generated by $\mcT$.

\begin{lemma}\label{prop:towerrefinement}
Fix a KR-partition $\mcT$, with at least three columns, and two clopen subsets $U\sim_K V$. There exists a KR-partition $\mcS$ such that 
\begin{enumerate}
\item $\mcS$ refines $\mcT$.
\item $\mcS$ is compatible with $U$ and $V$.
\item \label{eq:refinement} In each column of $\mcS$ there are as many atoms contained in $U$ as atoms contained in $V$.
\item \label{eq:notwins} All columns of $\mcS$ have the same height, each column of $\mcS$ contains at least one copy of every column of $\mcT$, and no two of them  are $\mcT$-twins.
\end{enumerate}
\end{lemma}

\begin{proof}
We may and do assume that $U,V$ are neither empty nor the whole $X$, and that they are disjoint. By cutting along the columns of $\mcT$, one can make sure that it is compatible with both $U$ and $V$, and this will remain true of any KR-partition refining it. So we assume that $\mcT$ satisfies this condition.

In each column $C$ of $\mcT$ which meets both $U$ and $V$, let $n_C(U)$ denotes the number of atoms of $C$ contained in $C$, and define similarly $n_C(V)$, and $m_C=\min(n_C(U),n_C(V))$. Shrink $U$ to $U'$ by removing from $U$ $m_C$ atoms of $C$ contained in $U$ in each column $C$, and similarly shrink $V$ to $V'$. If $U'$ or $V'$ is empty then our partition $\mcT$ already satisfies $\eqref{eq:refinement}$ and the proof is concluded by applying Lemma \ref{l:refinement}. Assume this is not the case; we have to find a refinement $\mcS$ of $\mcT$ such that \eqref{eq:refinement} and \eqref{eq:notwins} hold with $U'$ and $V'$ in place of $U,V$. That is, we have reduced to the case where each column of $\mcT$ meets at most one of $U$ or $V$. We now assume that we are in that situation.

We consider two cases. First, it might happen that one column, say $C_1$, of $\mcT$ meets neither $U$ nor $V$. Letting $C_2,\ldots,C_n$ denote the other columns of $\mcT$, they form a KR-partition $\mcT_Y$ of some clopen $Y$, in which $U$ and $V$ are contained; applying \cite{Ibarlucia2016}*{Proposition~3.5} to this partition, we refine $\mcT_Y$ to a new KR-partition $\mcT_Y'$ of $Y$ such that \eqref{eq:refinement} is satisfied. Adjoining $C_1$ to $\mcT_Y'$, we obtain a KR-partition of $X$ satisfying $\eqref{eq:refinement}$, and by construction $C_1$ has no $\mcT$-twin in $\mcT_Y'$. Thus we conclude by applying Lemma \ref{l:refinement} to this KR-partition.

The remaining case is when each column of $\mcT$ meets either $U$ or $V$. Since $\mcT$ is assumed to have at least three columns, we may assume w.l.o.g that $C_1$ and $C_2$ meet $U$, while $C_3$ meets $V$. Say that $C_1$ has $n_1$ atoms in $U$, $C_2$ has $n_2$ atoms in $U$, and $C_3$ has $n_3$ atoms in $V$. Then one may form a new KR-partition, with one column formed of $n_3$ copies of $C_1$ stacked onto $n_1$ copies of $C_3$, another consisting of $n_3$ copies of $C_2$ stacked onto $n_2$ copies of $C_3$, and the other columns $C_1',\ldots,C_n'$ being copies of $C_1,\ldots,C_n$. In this new KR-partition, we have two columns $D_0,D_1$ which are not $\mcT$-twins and which contain as many atoms in $U$ as in $V$; set these two columns apart, and remove the corresponding parts of $U,V$ to form $U'$, $V'$. Then $C_1',\ldots,C_n'$ form a KR-partition of some clopen $Y$, in which $U'$ and $V'$ are contained, and $U'$, $V'$ have equal measures. By applying \cite{Ibarlucia2016}*{Proposition~3.5}, we thus find a KR-partition of $Y$ which satisfies \eqref{eq:refinement}, and adjoining $D_0$, $D_1$ to this KR-partition yields a KR-partition satisfying \eqref{eq:refinement} and with two columns which are not $\mcT$-twins. We conclude by applying Lemma \ref{l:refinement} to this KR-partition.
\end{proof}

Using our previous lemmas, we may form a sequence of KR-partitions $\mcT_n$, with columns $(C^n_1,\ldots,C^n_{k_n})$ such that:
\begin{enumerate}
\item $\mcT_0$ consists of two atoms $A,B$ (i.e. there are only two columns, each of height $1$).
\item For all $n \ge 1$, each $\mcT_n$ has at least $3$ columns.
\item The diameter of the base and top of $\mcT_n$ converge to $0$.
\item Given any clopen $U,V$ such that $U \sim_K V$, there exists $n$ such that $\mcT_n$ is compatible with $U,V$ and each column of $\mcT_n$ has as many atoms contained in $U$ as atoms contained in $V$.
\item All columns of $\mcT_n$ have the same height, each column of $\mcT_{n+1}$ contains at least one copy of every column of $\mcT_n$, and no two of them are $\mcT_n$-twins.
\item For any column $C$ of $\mcT_{n+1}$, the ordering of levels of $C$ is such that the copies of $C^n_1$ contained in $C$ come first, followed by the copies of $C^n_2$, and so on.
\end{enumerate}

Note that $A,B$ above may be any two disjoint, nonempty clopen subsets partitioning the ambient Cantor space (and we may use the same $A$, $B$ when applying our construction to any divisible dynamical simplex). For those who are more used to thinking in terms of Bratteli diagrams, we note that the Bratteli diagram that we built above is both simple and left-ordered.

There exists a unique homeomorphism $\varphi$ of $X$ which extends all partial automorphims associated to $\mcT_n$, and the first four conditions above imply that $\varphi$ is minimal and $K=\{\mu \colon \varphi_* \mu=\mu\}$ (see \cite{Ibarlucia2016}*{Proposition~3.6 and Corollary~4.3}). It remains to prove that $\varphi$ is a Toeplitz subshift on the alphabet $\{0,1\}$. The main step is to prove that $A,B$ form a generating partition.

\begin{prop}
For each $x \ne y \in X$, there exists $k$ such that $\varphi^k(x) \in A$ and $\varphi^k(y) \in B$.
\end{prop}

\begin{proof}
Fix $x \ne y$. Let $m$ be the smallest integer such that there exists $k$ for which $\varphi^k(x)$ and $\varphi^k(y)$ belong to different atoms of $\mcT_m$; we want to prove that $m=0$. So assume for a contradiction that $m=n+1$.

We may as well assume that $x,y$ belong to different atoms of $\mcT_m$. Denote by $i(x)$ the smallest integer such that $\varphi^{-i(x)}(x)$ belongs to the base of $\mcT_n$, and similarly for $i(y)$ (thus $i(x)$ measures how far the atom containing $x$ is from the base of its column). 

We distinguish two cases: first, assume that $i(x)=i(y)$. Then, replacing $x,y$ by $\varphi^{-i(x)}(x)$, $\varphi^{-i(y)}(y)$ respectively, we have to deal with the case where $x$, $y$ belong to the base of $\mcT_m$, necessarily in different columns. Since no two columns of $\mcT_m$ are $\mcT_n$-twins, there exists a positive $j$ such that $\varphi^j(x)$ and $\varphi^j(y)$ belong to different $\mcT_n$-columns, contradicting the minimality of $m$.

The remaining case is that when $i(x) \ne i(y)$; find a column $C$ of $\mcT_m$ which contains the largest number of copies of $C^n_1$, and $j$ such that $\varphi^j(x)$ belongs to the base of $C$ (such a $j$ exists because we already know that $\varphi$ is minimal). Since $i(x) \ne i(y)$, and all columns of $\mcT_m$ have the same height, $\varphi^j(y)$ cannot belong to the base of $\mcT_m$. Still, $\varphi^j(y)$ must belong to $C^n_1$ by our assumption on $m$. But then the smallest positive $l$ such that $\varphi^{l+j}(y) \in C^n_2$ (which happens inside the same column of $\mcT_m$ as that which contains $\varphi^j(y)$, since each column of $\mcT_m$ contains at least one copy of each column of $\mcT_n$) must be such that $\varphi^{j+l}(x) \in C^n_1$, a contradiction.
\end{proof}

\begin{proof}[End of the proof of Theorem \ref{t:Toeplitz}]

Now, define $g \colon X \to \{0,1\}^\Z$ by setting $g(x)(n)=0 \leftrightarrow \varphi^n(x) \in A$. This is an embedding of $(X,\varphi)$ into $(\{0,1\}^\Z,S)$, where $S$ is the shift map - indeed, clearly $g$ is continuous and equivariant, and the previous proposition precisely asserts that $g$ is injective.

It remains to prove that $g(X)$ is Toeplitz. Denote by $x_{\infty}$ the intersection of the bases of $\mcT_n$. Let $B_n$ be the basis of $\mcT_n$, and $N$ the common height of all columns of $\mcT_n$. Then we have $\varphi^{Np}(B_n)=B_n$ for all $p \in \Z$. Since the sequence $g(B_n)$ forms a neighborhood basis for $g(x_\infty)$, this proves that $g(x_\infty)$ is quasiperiodic. As $\varphi$ is minimal, $g(x_\infty)$ is not periodic, so it is a Toeplitz sequence and we are done. This concludes the proof. \end{proof}

\begin{remark*} If one is willing to increase the number of blackboxes being used, it is actually very simple to deduce theorem \ref{t:Toeplitz} from \cite{Ibarlucia2016}: simply note that, if $K$ is a divisible dynamical simplex, then any KR-partition can be refined by a further KR-partition, all of whose columns have the same height (this is the last part of the proof of Lemma \ref{l:refinement}); then use a theorem of Sugisaki \cite{Sugisaki2001}*{Theorem~1.2} to conclude that the minimal homeomorphism produced by the construction of \cite{Ibarlucia2016}, with the additional condition that all columns of the sequence of partitions used in the construction have the same height, is strongly orbit equivalent to a Toeplitz subshift. However, the construction of \cite{Sugisaki2001} is fairly technical and dependent on Giordano--Putnam--Skau's theory, which is much less elementary than our cutting and stacking arguments above. Thus we feel it is worth going to the trouble of detailing our elementary argument.
\end{remark*}

\begin{question} Can one give a simple characterization of the dynamical simplices $K$ for which there exists a Toeplitz subshift $\varphi$ such that $K$ is the set of all $\varphi$-invariant Borel probability measures? 
\end{question}
Theorem \ref{t:Toeplitz} amounts to the statement that divisibility is a sufficient condition.

\section{Orbit equivalence and isomorphism of dynamical simplices}

In this section we go over some basic descriptive set-theoretic facts (namely, checking that certain sets and maps are Borel) and explain why orbit equivalence of minimal homeomorphisms can be recast as isomorphism of dynamical simplices. We will make use of the \emph{Effros Borel structure} on the set $\mcF(\Prob(X))$ made up of all \emph{nonempty} closed subsets of $\Prob(X)$; this is the $\sigma$-algebra generated by all sets of the form 
$$\{F \in \mcF(\Prob(X)) \colon F \cap O \ne \emptyset\} $$
where $O$ ranges over all open subsets of $\Prob(X)$. Equivalently, this is the $\sigma$-algebra of all Borel sets for the Vietoris topology on $\mcF(\Prob(X))$, and $\mcF(\Prob(X))$ endowed with the Effros Borel structure is a standard Borel space. 

The Kuratowski--Ryll-Nardzewski theorem allows us to fix for the remainder of this section a sequence of Borel maps $\mu_n \colon \mcF(\Prob(X)) \to \Prob(X)$ such that $\{\mu_n(K)\}$ is dense in $K$ for any $K \in \mcF(\Prob(X))$. For further details on the Vietoris topology, the Effros Borel structure and related results we refer the reader to \cite{Kechris1995}. The following lemma is well-known and appears for instance in \cite{Foreman2000}.

\begin{lemma}
Let $X$ be a Cantor space. The set $\Min(X)$ of minimal homeomorphisms of $X$ is a $G_\delta$ subset of the Polish group $\Homeo(X)$.
\end{lemma}

\begin{proof}
A homeomorphism $\varphi$ is minimal iff $X$ has no nontrivial invariant open subset, which is the same as saying that for any nonempty open subset $U$ one has $X= \bigcup_{n \in \Z} \varphi^n(U)$. By compactness of $X$ and bijectivity of $\varphi$, this is the same as saying that there exists $N  \ge 0$ such that $X= \bigcup_{n=0}^N \varphi^n(U)$. We may restrict our attention to clopen $U$, since those form a basis; this yields the equality
$$\Min(X) = \bigcap_{U \in \Clopen(X)} \bigcup_{N \in \N} \{\varphi \in \Homeo(X) \colon \bigcup_{n=0}^N \varphi^n(U)=X\}\ . $$
Each subset $\{\varphi \in \Homeo(X) \colon \bigcup_{n=0}^N \varphi^n(U)=X\}$ is open by definition of the topology on $\Homeo(X)$, proving that $\Min(X)$ is a $G_\delta$ subset of $\Homeo(X)$.
\end{proof}

\begin{lemma}
Let $X$ be a Cantor space. Given $\varphi \in \Homeo(X)$, let $K_\varphi$ denote the set $\{\mu \in \Prob(X) \colon \varphi_* \mu=\mu\}$. Then the map $\varphi \mapsto K_\varphi$ is a Borel map from $\Homeo(X)$ to $\mcF(\Prob(X))$.
\end{lemma}

\begin{proof}
We have to prove that, for any open subset $O \subset \Prob(X)$, the set $\{\varphi \in \Homeo(X) \colon K_\varphi \cap O \ne \emptyset \}$ is Borel.
By definition of the topology on $\Prob(X)$, and the fact that any open interval is a countable union of closed subintervals, it is sufficient to show that for any integer $n$, any clopen sets $U_1,\ldots,U_n$ and any closed intervals $I_0,\ldots,I_n$, the set 
$$A=\{\varphi \in \Homeo(X) \colon \exists \mu \in K_\varphi \ \forall j \in \{0,\ldots,n \} \ \mu(U_j) \in I_j\}$$
is Borel. It turns out that $A$ is actually closed in $\Homeo(X)$. To prove this, assume that $\varphi_i \in A$ converges to some $\varphi \in \Homeo(X)$, and let $\mu_i$ be $\varphi_i$-invariant measures such that $\mu_i(U_j) \in I_j$ for all $j \in \{0,\ldots,n\}$. Since $\Prob(X)$ is compact, we may assume that $\mu_i$ converges to some $\mu \in \Prob(X)$. For any clopen $U$ of $X$, $\varphi(U)$ is clopen, hence $\mu(\varphi(U))=\lim_i \mu_i(\varphi(U))$; but $\varphi(U)=\varphi_i(U)$ for all $i$ large enough, from which we obtain the equality
$$\mu(\varphi(U))= \lim_i \mu_i(\varphi_i(U))=\mu(U) \ .$$
This proves that $\mu \in K_\varphi$. Since for all $j$ we have $\mu(U_j)=\lim_i \mu_i(U_j)$ we also have that $\mu(U_j) \in I_j$, so $\varphi \in A$.
\end{proof}

\begin{lemma}
Let $X$ be a Cantor space. Then the set $Dyn(X)$ of all dynamical simplices on $X$ is a Borel subset of $\mcF(\Prob(X))$.
\end{lemma}

\begin{proof}
Fix a distance $d$ inducing the topology of $\Prob(X)$. Then $K$ is convex if, and only if, it satisfies the following condition:
$$\forall \varepsilon \in\Q^+ \ \forall n,m \in \N \ \exists p \in \N \ d\left(\frac{\mu_n(K) + \mu_m(K)}{2},\mu_p(K)\right) < \varepsilon\ .  $$
This shows that being convex is a Borel condition.

Saying that all elements of $K$ have full support is equivalent (by compactness of $K$) to stating that 
$$\forall U \in \Clopen(X) \setminus \{\emptyset\} \ \exists \varepsilon \in \Q^+ \  \forall n \in \N \  \mu_n(K)(U) \ge \varepsilon\ . $$
Since there are countably many clopen subsets of $X$ this is Borel.

Next we prove that the Glasner--Weiss property is Borel. Indeed, a Borel statement equivalent to this property is the assertion that, for all clopen $U,V$ and all $\varepsilon >0$, either there is some $n$ such that $\mu_n(K)(U) \ge \mu_n(K)(V)- \varepsilon$ or 
$$ \exists W \in \Clopen(X) \ W \subset V \text{ and } \forall n \in \N \ \mu_n(K)(U)= \mu_n(K)(W) \ .  $$
Above we are implicitly using compactness of $K$ and continuity of the maps $\mu \mapsto \mu(U)$ to deduce that the condition $\mu(U)< \mu(V)$ for all $\mu \in K$ is equivalent to saying that there is some $\varepsilon >0$ such that $\mu(U) < \mu(V) - \varepsilon$ for all $\mu \in K$.

Assuming that $K$ is convex, all its elements have full support and $K$ has the Glasner--Weiss property, the fact that all elements of $K$ are atomless is equivalent to the statement that
$$\forall \varepsilon \in \Q^+ \ \forall U \in \Clopen(X) \ \exists V \in \Clopen(X) \ \ V\subseteq U \text{ and } \forall n \ \mu_n(K)(V) \le \varepsilon \ .  $$

\end{proof}

\begin{lemma}
let $X$ be a Cantor space. The space of divisible subsets is Borel in $\mcF(\Prob(X))$.
\end{lemma}

\begin{proof}
Simply note that $K \in \mcF(\Prob(X))$ is divisible if and only if 
$$\forall n \in \N^* \ \forall U \in \Clopen(X) \exists V \in \Clopen(X) \ V\subset U \text{ and } \forall i \ \mu_i(K)(V) = \frac{1}{n}\mu_i(K)(U) \ .  $$
\end{proof}

We recall some definitions given in the introduction.

\begin{defn}
Let $X$ be a Cantor space. We say that two dynamical simplices $K,L$ on $X$ are \emph{isomorphic} if there exists some $g \in \Homeo(X)$ such that 
$g_*K=L$. 
\end{defn}

\begin{defn}
Let $X$ be a Cantor space, and $\varphi,\psi$ be two homeomorphisms of $X$. Denote by $R_\varphi, R_\psi$ the equivalence relations corresponding to the orbit partitions associated to $\varphi, \psi$. We say that $\varphi$ and $\psi$ are \emph{orbit equivalent} if there exists some $g \in \Homeo(X)$ such that
$$\forall x,y \in X (x R_\varphi y ) \Leftrightarrow (g(x) R_\psi g(y) )\ . $$
\end{defn}

\begin{prop}
The relations of orbit equivalence of minimal homeomorphisms and isomorphism of dynamical simplices are Borel bireducible.
\end{prop}

\begin{proof}
We already know that $\varphi \mapsto K_\varphi$ is Borel. The Giordano--Putnam--Skau theorem recalled in the introduction is exactly the statement that two minimal homeomorphisms are orbit equivalent if and only if $K_\varphi$ and $K_\psi$ are isomorphic dynamical simplices.
Hence $\varphi \mapsto K_\varphi$ is a Borel reduction from OE to isomorphism of dynamical simplices.

Conversely, the construction in \cite{Ibarlucia2016} associates a minimal homeomorphism to any dynamical simplex; this construction can be turned into a Borel map $K \mapsto \varphi_K$, since it involves building a neighborhood basis of $\varphi_K$ via an inductive construction where at each step one can simply choose the first witness that a certain Borel condition is satisfied. Using the Giordano--Putnam--Skau theorem again, we see that $K \mapsto \varphi_K$ reduces isomorphism of dynamical simplices to OE.
\end{proof}

\begin{remark*} In the construction of \cite{Ibarlucia2016}, the homeomorphisms $\varphi_K$ are actually saturated (i.e.~the topological full group of $\varphi_K$ is dense in its full group), so that $\varphi_K$ and $\varphi_L$ are orbit equivalent iff they are strong orbit equivalent (as we focus on orbit equivalence here, we do not give details). Thus the argument above also shows that isomorphism of dynamical simplices Borel reduces to strong orbit equivalence of minimal homeomorphisms. Since strong orbit equivalence is classifiable by countable structures, the main result of the next section will also establish that strong orbit equivalence of Toeplitz subshifts is $S_\infty$-universal. In an attempt at brevity, we will not elaborate more on this.
\end{remark*}

\section{Reducing homeomorphism of $0$-dimensional compact metrizable spaces to isomorphism of divisible dynamical simplices }
In this section, we explain the construction of \cite{Melleray2018} and how to apply it in order to build a Borel reduction from the relation of homeomorphism between $0$-dimensional compact metric spaces to the relation of isomorphism of divisible dynamical simplices.

Given a Choquet simplex $Q$, we denote by $\Aff(Q)$ the set of continuous, real-valued affine functions on $Q$; for $F \subseteq \Aff(Q)$ we denote by $F^+$ the elements of $F$ taking only positive values and by $F_1^+$ the set of elements of $F^+$ having all their values smaller than $1$ (that is, the intersection of $F^+$ with the unit ball for the supremum norm).

\begin{defn}
A subset $F$ of $\Aff(Q)$ is said to have the \emph{finite sum property} if for any $f_1,\ldots,f_n$, $g_1,\ldots,g_m \in F^+$ such that 
$\sum_{i=1}^n f_i= \sum_{j=1}^m g_j$ one can find $h_{i,j} \in F^+$ satisfying
$$\forall i \in \{1,\ldots,n\} \ f_i = \sum_{j=1}^m h_{i,j} \quad \text{ and} \quad \forall j \in \{1,\ldots,m\} g_j=\sum_{i=1}^n h_{i,j}\ .$$
\end{defn}

Whenever $Q$ is a Choquet simplex, $\Aff(Q)$ itself satisfies the finite sum property (see for instance \cite{Lussky1981}). Below we will make use of a specific example, so no knowledge of the theory of Choquet simplices is required. Assume that $X$ is a $0$-dimensional compact metrisable space; then $X$ is naturally identified with the extreme boundary of the Choquet simplex $\Prob(X)$, and every continuous function of $X$ extends uniquely to a continuous affine function on $\Prob(X)$. Denote by $F(X)$ the set of continuous affine functions on $\Prob(X)$ whose restriction to $X$ takes finitely many rational values; it is  straightforward to check that $F(X)$ is a countable dense subset of $\Prob(X)$ containing the constant functions and satisfying the finite sum property. Denote $G(X)={F(X)}_1^+$.

We note now some key properties of $G(X)$, which are easy to establish.

\begin{prop}\label{p:propsG(X)}
Let $K$ be a Cantor space, and $X$ be a closed subset of $K$. Then:
\begin{enumerate}
\item $F(X)$ satisfies the finite sum property.
\item For any $f_1,f_2 \in G(K)$ and any $f \in G(X)$ such that ${f_1}_{|X} \le f \le {f_2}_{|X}$, there exists $g \in G(K)$ extending $f$ and such that $f_1 \le g \le f_2$.
\item For any $f_1,\ldots,f \in G(X)$ and any $f \in G(K)$ such that 
$f_{|X} = \sum_{i=1}^n f_i $, there exists $g_1,\ldots,g_n \in G(K)$ extending $f_1,\ldots,f_n$ and such that $f=\sum_{i=1}^n g_i$.
\end{enumerate}
\end{prop}

\begin{proof} The proofs are easy so we try not to belabor the point.
\begin{enumerate}
\item Pick $f_1,\ldots,f_n$, $g_1,\ldots,g_m \in F(X)^+$ such that $\sum_{i=1}^n f_i = \sum_{j=1}^m g_j$. We may find a partition of $K$ by clopen sets $U_1,\ldots,U_N$ such that each $f_i$ and each $g_j$ are constant on $U_l$ for all $l$. Picking $x_1,\ldots,x_l  \in U_l$, it is an easy task to find rationals $q_{i,j}(l) >0$ such that for all $j$ one has $\sum_{i=1}^n q_{i,j}(l)=g_j(x_l)$, and for all $i$ $\sum_{j=1}^m q_{i,j}(l)=f_i(x_l)$. Then setting $h_{i,j}(x)=q_{i,j}(l)$ for each $x \in U_l$ works. 
\item We may find finitely many disjoint clopen sets $U_1,\ldots,U_N$ covering $K$ such that $f_1,f_2$ are constant on each $U_l$ and $f$ is constant on each $U_l \cap X$. For any $l$ such that $U_l \cap X \ne \emptyset$, we pick $x_l \in U_l \cap X$ and define $g$ to be equal to $f(x_l)$ on $U_l$. For any other $l$ we define $g$ to be equal to $f_2$ on $U_l$.
\item This follows easily from the previous fact (and is proved in exactly the same way as \cite{Melleray2018}*{Lemma~5.11}).
\end{enumerate}

\end{proof}

The construction of \cite{Melleray2018} takes as input a nonempty metrizable Choquet simplex $Q$, along with a countable, dense $\Q$-vector subspace $F$ of $\Aff(Q)$ containing $1$ and having the finite sum property; and yields as output a dynamical simplex $S=\{\mu_q\}_{q \in Q}$ affinely homeomorphic to $Q$, and (denoting by $K$ the underlying Cantor space of $S$) such that 
$$\{q \mapsto \mu_q(A)\}_{A \in \Clopen(K)}=F_1^+ \ .$$
The idea here is to start from a nonempty metrizable $0$-dimensional compact space $X$, and to apply that construction to $(\Prob(X),G(X))$ in order to produce a dynamical simplex $S(X)$. If $X$ and $X'$ are homeomorphic then $S(X)$ and $S(X')$ will be isomorphic; and conversely if $S(X)$ and $S(X')$ are isomorphic then their extreme boundaries are homeomorphic, i.e.~$X$ and $X'$ are homeomorphic. 

We now need to give some more detail on the construction of \cite{Melleray2018}, in order to convince the reader that it has the properties mentioned in the previous paragraph, and that it can be encoded in a Borel way.

Given a nonempty compact metrizable space $X$, a $X$-\emph{structure} is an object of the form $(\mcA,(\mu_x)_{x \in X})$ such that

\begin{itemize}
\item $\mcA$ is Boolean algebra .
\item Each $\mu_x$ is a probability measure on $\mcA$. 
\end{itemize} 

We say that the structure is \emph{finite} (resp. countable) if its underlying Boolean algebra is finite (resp. countable).

Below we briefly discuss Fra\"iss\'e classes and limits. We refer to \cite{Melleray2018} and the references therein for more details about Fra\"iss\'e classes. 
Fix a nonempty compact metrizable set $X$ for the duration of our discussion of Fra\"iss\'e classes.

\begin{defn}
The \emph{age} of a $X$-structure $A$ is the class of all finite $X$-structures which embed in $A$.
\end{defn}

A class $\mcL$ of finite $X$-structures is a \emph{Fra\"iss\'e class} if:
\begin{itemize}
\item It contains only countably many elements up to isomorphism.
\item For any $A,B \in \mcL$ there exists $C \in \mcL$ such that both $A$ and $B$ embed in $\mcL$.
\item Any substructure of an element of $\mcL$ also belongs to $\mcL$.
\item For any $A,B,C \in \mcL$ and any embeddings $i \colon \colon A \to B$, $j \colon A \to C$ there exists $D \in \mcL$ and embeddings $i' \circ B \to D$, $j' \circ C \to D$ such that $i' \circ i = j' \circ j$.
\end{itemize}

The last property above, known as the \emph{amalgamation property}, is the strongest and typically hardest to prove. It characterizes \emph{ultrahomogeneous structures}.

\begin{defn}
A $X$-structure $A$ is \emph{ultrahomogeneous} if any partial isomorphism of $A$ with domain a finite substructure extends to an automorphism of $A$.
\end{defn}

\begin{theorem}
The age of a ultrahomogeneous $X$-structure is a Fra\"iss\'e class; conversely, for any Fra\"iss\'e class $\mcL$ there exists a countable ultrahomogeneous $X$-structure whose age is equal to $\mcL$. This structure is unique (up to isomorphism) and is known as the \emph{Fra\"iss\'e limit} of $\mcL$. 
\end{theorem}

Not every structure whose age is a Fra\"iss\'e class is ultrahomogeneous; but among structures whose age is a Fra\"iss\'e class $\mcL$, those which are isomorphic to the Fra\"iss\'e limit of $\mcL$ are easy to recognize.

\begin{theorem}
Assume that $\mcA$ is a $X$-structure whose age is a Fra\"iss\'e class $\mcL$. Then $\mcA$ is isomorphic to the Fra\"iss\'e limit of $\mcL$ if and only if for any finite substructure $\mcB$ of $\mcA$, and any embedding $\alpha$ from $\mcB$ to some $\mcC \in \mcL$, there exists an embedding $\beta \colon \mcC \to \mcA$ such that $\beta \circ \alpha(a)=a$ for all $a \in A$.
\end{theorem}

Now we can explain the construction of \cite{Melleray2018}. Given a $0$-dimensional compact metrizable space $X$, we may consider the class $\mcL_X$ of all finite $X$-structures ($\mcA,(\mu_x)_{x \in X})$ such that for each nonzero $a \in \mcA$, $x \mapsto \mu_x(a)$ belongs to $G(X)$.

It follows from the arguments of \cite{Melleray2018} and the finite sum property of $G(X)$ that $\mcL_X$ is a Fra\"iss\'e class.  Its limit is of the form $(\mcA,(\mu_x)_{x \in X})$, where $\mcA$ is an infinite countable atomless Boolean algebra. Denoting by $C$ the Stone dual of $\mcA$, it is further established in \cite{Melleray2018} that the map $x \mapsto \mu_x$ is a continuous embedding of $X$ into $\Prob(C)$, that the closed convex hull of $\{\mu_x\}_{x \in X}$ is a (divisible) dynamical simplex and that $G(X)$ coincides with the set of all maps $x \mapsto \mu_x(a)$ as $a$ ranges over all nonzero elements of $\mcA$. Further, each $\mu_x$ is an extreme point of that dynamical simplex, from which it follows that its extreme boundary coincides with $\{\mu_x\}_{x \in X}$ (and is thus homeomorphic to $X$).
 We just realized $\Prob(X)$ as a dynamical simplex in $\Prob(C)$, in such a way that $G(X)$ coincides with all maps $x \mapsto \mu_x(U)$ as $U$ runs over all nonempty clopen subsets of $X$. 

Fix a Cantor space $K$, and apply the procedure we just described to $K$. This yields a Cantor space $C$, and a continuous map $k \mapsto \mu_k$ from $K$ to $\Prob(C)$ such that the closed convex hull of $\{\mu_k \colon k \in K\}$ is a divisible dynamical simplex with extreme boundary homeomorphic to $K$.

\begin{defn}
For any nonempty closed subset $X$ of $K$, denote by $S(X)$ the closed convex hull of $\{\mu_x \colon x \in X\}$.
\end{defn}

The map $X \mapsto S(X)$ is a continuous map from $\mcF(X)$ to $\mcF(\Prob(C))$.

\begin{prop} The following facts hold.
\begin{enumerate}
\item\label{p:ultrahomogeneous} For any closed nonempty subset $X$ of $K$, the $X$-structure $(\Clopen(C),(\mu_x)_{x \in X})$ is ultrahomogeneous.
\item\label{p:divisiblesimplex} For any closed nonempty subset $X$ of $K$, $S(X)$ is a divisible dynamical simplex.
\item\label{p:reduction} $S(X)$ and $S(X')$ are isomorphic iff $X$ and $X'$ are homeomorphic.
\end{enumerate}
\end{prop}

\begin{proof} \eqref{p:ultrahomogeneous} The argument is very similar to arguments of \cite{Melleray2018}. We write it down for the reader's convenience.
  Let $\mcA$ be a finite subalgebra of $\Clopen(C)$, and assume that $\alpha$ is an embedding of $(\mcA,(\mu_x)_{x \in X})$ in a finite $X$-structure $(\mcB,(\nu_x))$. For any atom $a$ of $\mcA$, let $(b_i^a)_{i \in I_a}$ denote the atoms of $\mcB$ which are contained in $\alpha(a)$. Denote by $f_a$ the map $k \mapsto \mu_k(a)$ (defined on the whole $K$), and by $g_b$ the map $x \mapsto \mu_x(b)$ (defined only on $X$). Then we have for any atom $a$ of $\mcA$ that 
$$\forall x \in X \ \sum_{b \in I_a} g_b= f_a\ . $$
By Proposition \ref{p:propsG(X)}, we may extend each $g_b$ to a map $h_b \in G(K)$ such that 
$$\forall k \in K \ \sum_{b \in I_a} h_b(k)= f_a(k) $$
Using the Fra\"iss\'e property of $(\Clopen(X),(\mu_k)_{k \in K})$, we can find clopen subsets $(U_b^a)_{b \in I_a}$ of $\Clopen(C)$ such that 
$a= \bigcup_{b \in I_a} U_b^a$ for any atom $a$ of $\mcA$, and $\mu_k(U_b^a)=h_b(k)$ for all $k \in K$. In particular, $\mu_x(U_a^b)=\nu_x(b)$ for all $x \in X$ and all $b$, so these sets witness the fact that $(\Clopen(C),(\mu_x)_{x \in X})$ satisfies the Fra\"iss\'e property.

\eqref{p:divisiblesimplex} Since $S(K)$ is divisible, it is immediate that $S(X)$ is also divisible. It follows from Proposition \ref{p:propsG(X)} that the age of $(\Clopen(X),(\mu_x)_{x \in X})$ is equal to $\mcL_X$. Hence \cite{Melleray2018}*{Proposition~3.9} ensures that $S(X)$ is a dynamical simplex.

\eqref{p:reduction} Assume that $S(X)$ and $S(X')$ are isomorphic. Then their extreme boundaries are homeomorphic, that is, $X$ and $X'$ are homeomorphic. 
Conversely, let $\varphi \colon X \to X'$ be a homeomorphism. Then $(\Clopen(C),(\mu_{\varphi(x)})_{x \in X})$ is an ultrahomogeneous $X$-structure, with the same age as the ultrahomogeneous $X$-structure $(\Clopen(C),(\mu_x)_{x \in X})$. Thus there exists an automorphism $h$ of $\Clopen(C)$, equivalently a homeomorphism $h$ of $C$, such that $h_*\mu_x= \mu_{\varphi(x)}$ for all $x \in X$. We then have $h_*S(X)=S(X')$.
\end{proof}

We are finally done.

\begin{theorem}
The following equivalence relations are Borel bireducible.
\begin{enumerate}
\item\label{OE} Orbit equivalence of minimal homeomorphisms;
\item\label{OE-T} Orbit equivalence of Toeplitz subshifts;
\item\label{Iso} Isomorphism of dynamical simplices;
\item\label{Iso-D} Isomorphism of divisible dynamical simplices;
\item\label{Homeo} Homeomorphism of closed subsets of the Cantor space.
\end{enumerate}
\end{theorem}

\begin{proof}
We already know that \eqref{OE} and \eqref{Iso} are Borel bireducible. A Borel coding of the construction used in the proof of \ref{t:Toeplitz} (which, as we already mentioned, is easy to obtain) produces a Borel reduction of \eqref{Iso-D} to \eqref{OE-T}. Clearly \eqref{OE-T} Borel reduces to \eqref{OE}. Since \eqref{Iso} is induced by a Borel action of a closed subgroup of $S_\infty$, it must Borel reduce to \eqref{Homeo} by the theorem of Camerlo and Gao mentioned in the introduction. Finally, the map $X \mapsto S(X)$ yields a continuous reduction of \eqref{Homeo} to \eqref{Iso-D}. We thus established the existence of the following Borel reductions:
$$\eqref{Iso-D} \preceq \eqref{OE-T} \preceq \eqref{OE} \preceq \eqref{Iso} \preceq \eqref{Homeo} \preceq \eqref{Iso-D} \ .$$
\end{proof}

We again note without further details that strong orbit equivalence (of minimal homeomorphisms, or Toeplitz subshifts) also sits at the same complexity level.


\bibliography{mybiblio}

\end{document}